\documentclass{amsart} 
\usepackage{amssymb}
\usepackage{amsthm}
\usepackage{dsfont}
\usepackage{xcolor}
\usepackage{cleveref}
\usepackage[curve,matrix,arrow]{xy}
\theoremstyle{plain}\newtheorem{Theorem}{Theorem}[section]
\theoremstyle{plain}
\theoremstyle{plain}\newtheorem{Corollary}[Theorem]{Corollary}
\theoremstyle{plain}\newtheorem{Lemma}[Theorem]{Lemma}
\theoremstyle{plain}\newtheorem{Proposition}[Theorem]{Proposition}
\theoremstyle{definition}
\theoremstyle{definition}\newtheorem{Example}[Theorem]{Example}
\theoremstyle{definition}
\theoremstyle{definition}\newtheorem{Remark}[Theorem]{Remark}
\theoremstyle{definition}
\theoremstyle{plain}\newtheorem{Statement}[Theorem]{}

\def\F{{\mathbb F}}
\def\Fp{{\mathbb F_p}}

\def\Aut{\mathrm{Aut}}           \def\tenk{\otimes_k}     
             \def\ten{\otimes} 

\def\dim{\mathrm{dim}}           
\def\Der{\mathrm{Der}}       
\def\End{\mathrm{End}}           

\def\Ext{\mathrm{Ext}}

\def\GL{\mathrm{GL}} \def\gl{\mathfrak{gl}}
\def\H{\mathrm H}
\def\HH{\mathrm{HH}}
\def\Hom{\mathrm{Hom}}           
  
\def\ker{\mathrm{ker}}           
             
\def\IDer{\mathrm{IDer}}
\def\Im{\mathrm{Im}}

\def\ll{\ell\!\ell}

\def\op{\mathrm{op}}

\def\Tr{\mathrm{Tr}}

\title{On the solvability of the Lie algebra  $\HH^1(B)$ for blocks of finite groups} 
\author{Markus Linckelmann, Jialin Wang} 

\address{Markus Linckelmann \\
School of Science \& Technology \\
Department of Mathematics \\
City St George's, University of London \\
Northampton Square \\
London EC1V 0HB \\
United Kingdom}
\email{markus.linckelmann.1@city.ac.uk}

\address{Jialin Wang \\
School of Science \& Technology \\
Department of Mathematics \\
City St George's, University of London \\
Northampton Square \\
London EC1V 0HB \\
United Kingdom}
\email{jialin.wang@city.ac.uk}

\begin{document}

\begin{abstract}
We give some criteria for the Lie algebra $\HH^1(B)$ to be solvable, where
$B$ is a $p$-block of a finite group algebra, in terms of the action of an
inertial quotient of $B$ on a defect group of $B$. 
\end{abstract}

\maketitle

\section{Introduction}

The Lie algebra structure of the first Hochschild cohomology of a block of a 
finite group algebra   sits at the crossroads of the representation theory of a block
as a part of the wider theory of representations of finite-dimensional algebras
and the  fusion systems and their invariants that can be associated  with block
algebras.  This Lie algebra is therefore one of the ingredients that has the potential
to feed into an understanding of the connections between the global and
local structure of block algebras. 
The purpose of the present paper is to contribute to  investigating this connection.

Let  $p$ be a prime number and $k$ a field of characteristic $p$. 
A block of a finite group algebra $kG$ is an indecomposable direct factor $B$
of $kG$ as an algebra. A defect group of a block $B$ of $kG$ is a maximal
$p$-subgroup $P$ of $G$ such that $kP$ is isomorphic to a direct summand of
$B$ as an $kP$-$kP$-bimodule. The results in this paper are a contribution to the
broader theme investigating 
connections between Hochschild cohomology and fusion systems of blocks.
More precisely, the main results  of this paper relate 
the Lie algebra structure of $\HH^1(B)$, notably the solvability of this Lie algebra,
to the action of an inertial quotient $E$ on a  defect group of the block.

For $P$ a finite $p$-group  we denote by $\Phi(P)$ the Frattini subgroup of $P$;
this is the smallest normal subgoup of $P$ such that $P/\Phi(P)$ is elementary 
abelian.  If $E$ is a finite group acting on $P$, then this action induces an action 
of $E$ on $P/\Phi(P)$. In this way we can regard $P/\Phi(P)$ as an $\Fp E$-module. 
If in addition $E$ has order prime to $p$, then $P/\Phi(P)$ is a semisimple 
$\Fp E$-module. The following results have in common that the property of this 
module being multiplicity free is the key ingredient for the first Hochschild 
cohomology to be solvable as a  Lie algebra.

\begin{Theorem} \label{frobenius-block-thm}
Let $G$ be a finite group, and assume that $k$ is large enough for the 
subgroups of  $G$. 
Let $B$ be a block of $kG$ with a non-trivial abelian defect group $P$ and
a non-trivial inertial quotient $E$ acting freely on $P\setminus \{1\}$.
If the  $\Fp E$-module $P/\Phi(P)$ is multiplicity free, then
the Lie algebra $\HH^1(B)$ is solvable. The converse holds if $p$ is odd.
\end{Theorem}

In the course of the proof we will describe more precise results on 
the Lie algebra structure of $\HH^1(B)$.
One key ingredient  is a stable equivalence of Morita type between
the block $B$ and the semidirect product $k(P\rtimes E)$, due to Puig.
Another key ingredient is the next result which investigates the
Lie algebra structure of $\HH^1(kP)^E$. We denote by $[P,E]$ the 
subgroup of $P$ generated by the set of elements of the form 
$(^eu)u^{-1}$, where $u\in P$ and $e\in E$ (this is the hyperfocal subgroup 
in $P$ of $P\rtimes E$). 

\begin{Theorem} \label{E-stable-derivations-thm}
Let $P$ be a non-trivial finite abelian $p$-group and $E$ a finite  $p'$-group 
acting on $P$.  Suppose  that $[P,E]=P$. 
\begin{itemize}
\item[{\rm (i)}]
Every $E$-stable derivation on $kP$ has image contained in the Jacobson 
radical  $J(kP)$. 
\item[{\rm (ii)}]
If the $\Fp E$-module $P/\Phi(P)$ is multiplicity free, then $\HH^1(kP)^E$ 
is a solvable Lie algebra.  The converse holds if $p$ is odd.
\end{itemize}
\end{Theorem}

The two theorems above will be proved in Section \ref{Proof-Section}.
When the acting $p'$-group $E$ is abelian as well, we can be more precise. See 
Section \ref{twisted-Section} for the notation and basic facts on twisted group
algebras. The following result will be proved in Section \ref{proof-section-2}.

\begin{Theorem} \label{twisted-PE-thm}
Let $P$ be a non-trivial finite abelian $p$-group and $E$ an abelian $p'$-subgroup 
of $\Aut(P)$. 
Let $\alpha\in$ $Z^2(E;k^\times)$ inflated to $P\rtimes E$ via the canonical 
surjection $P\rtimes E\to E$.  Suppose that  $[P,E]=P$. 
\begin{itemize}
\item[{\rm (i)}]
Every class in $\HH^1(k_\alpha(P\rtimes E))$ is represented by a derivation on 
$k_\alpha(P\rtimes E)$ with image contained in the Jacobson radical 
$J(k_\alpha(P\rtimes E))$.
\item[{\rm (ii)}]
If the   $\Fp E$-module $P/\Phi(P)$ is multiplicity free, then the Lie algebra 
$\HH^1(k_\alpha(P\rtimes E))$ is solvable. The converse holds if  $p$ is odd.
\end{itemize}
\end{Theorem} 

If one replaces twisted group algebras by group algebras of corresponding central
extensions, then Theorem \ref{twisted-PE-thm} admits an equivalent
reformulation, in which the acting group $E$ need not act faithfully and need
not be abelian so long as its image in $\Aut(P)$ is abelian; see Theorem
\ref{central-PE-thm} below. 
We illustrate the above results in conjunction with the
structure theory of normal defect blocks in Theorem \ref{normal-defect-thm}
and Corollary \ref{normal-defect-cor}, and we determine under what circumstances
the Lie algebra $\HH^1(B)$ is simple or solvable for blocks $B$ with elementary abelian
defect of rank $2$ and abelian inertial quotient in Example \ref{CpCp-example}.

\section{Background material} 

Let $k$ be a field.
Let $A$ be an associative unital $k$-algebra. A {\em derivation on $A$} is a $k$-linear 
map $f : A\to A$ satisfying the Leibniz rule $f(ab)=f(a)b+af(b)$, for all $a$, $b \in A$. 
The Leibniz rule implies that any derivation $f$ on $A$ vanishes at all central 
idempotents; in particular, $f(1)=0$.
The set $\Der(A)$ of all derivations on $A$ is a Lie subalgebra of $\End_k(A)$
with Lie bracket $[f,g]=f\circ g-g\circ f$, for all $f$, $g\in \End_k(A)$.
If $c\in A$, then the map $[c,-]$ sending $a\in A$ to the additive commutator
 $[c,a]=ca-ac$ is a derivation.
The derivations of this form are called {\em inner derivations on} $A$, and
the subspace $\IDer(A)$ of inner derivations is an ideal in the Lie algebra
$\Der(A)$. 

For $M$ an $A$-$A$-bimodule, regarded as an $A\tenk A^\op$-module, the
Hochschild cohomology of $A$ with coefficients in $M$ is the graded $k$-module
$\HH^*(A;M)=$ $\Ext_{A\tenk A^\op}^*(A; M)$. We set $\HH^*(A)=$ $\HH^*(A;A)$.
Then $\HH^*(A)$ is a graded-commutative algebra, and $\HH^*(A;M)$ is
a graded right $\HH^*(A)$-module.   We have canonical identifications
$\HH^0(A)\cong \mathrm{Z}(A)$ and $\HH^1(A)\cong \Der(A)/\IDer(A)$; see for 
instance Weibel \cite[Lemma 9.2.1]{Weibel}. If $f$ is a derivation on $A$ and 
$\alpha$ a $k$-algebra automorphism of $A$, then $\alpha^{-1} \circ f\circ \alpha$ 
is a derivation on $A$, and if $f$ is an inner derivation, then so is 
$\alpha^{-1}\circ f\circ \alpha$. Thus if $E$ is a group acting on $A$ by $k$-algebra
automorphisms, then this action induces an action of $E$ on $\HH^1(A)$ by 
Lie algebra automorphisms, and the subspace $\HH^1(A)^E$ of $E$-fixed points
in $\HH^1(A)$ is a Lie subalgebra of $\HH^1(A)$.  We will need the following
well-known facts. 

\begin{Lemma}[{cf. \cite[Lemma 3.1]{LiRusimple}}] \label{ZA-derivations}
Let $A$  be a finite-dimensional associative unital $k$-algebra.
For every derivation $f$ on $A$ we have $f(Z(A))\subseteq$ $Z(A)$.
\end{Lemma}

\begin{Lemma}[{cf. \cite[Lemma 2.4]{LiRuHH1} }] \label{Czero}
Let $A$ be a  finite-dimensional associative unital $k$-algebra. Suppose that $A$
has a separable subalgebra  $C$ such that $A=$ $C\oplus J(A)$. 
Every class in $\HH^1(A)$ has a representative $f\in$ $\Der(A)$
satisfying $C\subseteq$ $\ker(f)$.
\end{Lemma}

We note that in \cite[Lemma 2.4, Proposition 2.8]{LiRuHH1} the algebra $A$ is assumed 
to  be split, but the proof there shows that this is not needed so long as in the previous
Lemma $A$ is  assumed to have a  separable subalgebra $C$ satisfying $A=C\oplus J(A)$.  
By the  Malcev-Wedderburn Theorem, this is equivalent to requiring $A/J(A)$ to be 
separable,  in which case we have $C\cong A/J(A)$.
If $f$ is a derivation on $A$ which vanishes on $C$ and sends $J(A)$ to $J(A)^m$ for 
some positive integer $m$, then in fact $\Im(f)\subseteq J(A)^m$. The following 
Proposition is a  slight variation of \cite[Proposition 2.8]{LiRuHH1}, with essentially 
unchanged proofs, making repeatedly use of the Leibniz rule. 

\begin{Proposition} \label{rad-Prop}
Let $A$ be a  finite-dimensional associative unital $k$-algebra. 
For $m\geq 1$, denote by $\Der_m(A)$ the subspace of $\Der(A)$ consisting
of all derivations $f : A\to$ $A$ such that 
such that $\Im(f) \subseteq$ $J(A)^m$.
The following hold.

\begin{enumerate}
\item[\rm (i)]
For any positive integers $m$, $n$ we have $[\Der_m(A), \Der_n(A)]\subseteq$
$\Der_{m+n-1}(A)$. 

\item[\rm (ii)]
The space $\Der_1(A)$ is a Lie subalgebra of $\Der(A)$, and for any positive 
integer $m$, the space $\Der_m(A)$ is a Lie ideal in $\Der_1(A)$. 

\item[\rm (iii)]
The space $\Der_2(A)$ is a nilpotent ideal in $\Der_1(A)$. More precisely, if
$\ll(A)\leq 2$, then $\Der_2(A)=$ $0$, and if $\ll(A)>2$, then the
nilpotency class of $\Der_2(A)$ is at most $\ll(A)-2$. 
\end{enumerate}
\end{Proposition}

\begin{Corollary} \label{rad-Cor}
With the notation and hypotheses of Proposition \ref{rad-Prop}, the following hold.
\begin{itemize}
\item[{\rm (i)}]
If $[\Der_1(A),\Der_1(A)] \subseteq \Der_2(A)$, then $\Der_1(A)$ is a solvable
Lie algebra.
\item[{\rm (ii)}]
If $[\Der_1(A),\Der_1(A)]\subseteq \Der_2(A) + \IDer(A)$, then the image of 
$\Der_1(A)$ in $\HH^1(A)$ is a solvable Lie algebra.
\item[{\rm (iii)}] 
If the canonical map $\Der_1(A)\to \HH^1(A)$ is surjective, and if 
$[\Der_1(A),\Der_1(A)] \subseteq \Der_2(A) +\IDer(A)$, then $\HH^1(A)$ is a 
solvable  Lie algebra.
\item[{\rm (iv)}]
If the canonical map $\Der_1(A)\to \HH^1(A)$ is surjective, then the image of 
$\Der_2(A)$ is a nilpotent ideal in $ \HH^1(A)$. Furthermore, suppose 
$\HH^1(A)= L+D_2$ where $L$ is a Lie subalgebra and $D_2$ is the image of 
$\Der_2(A)$, then $\HH^1(A)$ is solvable if and only if $L$ is solvable.
\end{itemize} 
\end{Corollary}

\begin{proof}
The statements (i) and (ii)  follow from Proposition \ref{rad-Prop} (iii) and the 
assumptions.  Statement (iii) are immediate consequences of (ii). 
As for statement (iv), suppose $\HH^1(A)= L+D_2$ as in the statement. If $L$ is 
not solvable, then $ \HH^1(A)$ is not solvable. Suppose $L$ is solvable, and 
$L(n)=0$ for some positive integer $n$, where $L(n)$ denotes the $n$-th derived
Lie algebra of $L$. 
Since $D_2$ is an ideal, it follows that for any $i\geq 1$ we have
$ \HH^1(A)(i)\subseteq L(i)+D_2$. Thus, $\HH^1(A)(n)\subseteq D_2$. 
The statement follows since $D_2$ is nilpotent.
\end{proof}

We denote by $[A,A]$ the subspace spanned by the set  of additive commutators 
$[a,b]=ab-ba$, where $a$, $b \in A$. Since $[a,b]c=abc-bac=abc-acb+acb-bac=
a[b,c] + [ac,b]$ for all $a$, $b$, $c\in A$, we have $[A,A]A=$ $A[A,A]$, and this
is the smallest ideal such that the corresponding quotient of $A$ is commutative.

\begin{Lemma} \label{Der-Lemma1}
Let $A$ be a finite-dimensional associative unital $k$-algebra. Every derivation on $A$
preserves the subspace $[A,A]$ and the ideal $[A,A]A$, and induces a
derivation on $A/[A,A]A$. Under this correspondence, an inner derivation on
$A$ is mapped to zero. In particular, this correspondence induces a
Lie algebra homomorphism $\HH^1(A)\to \HH^1(A/[A,A]A)$.
\end{Lemma} 

\begin{proof}
If $f$ is a derivation on $A$, then $f([a,b])=$ $f(a)b+af(b)-f(b)a-bf(a)=$ 
$[f(a),b]+[a,f(b)]$, and hence $f$ preserves the subspace $[A,A]$ and  hence also the
ideal $[A,A]A$, using the Leibniz rule. Thus $f$ induces a derivation on 
$A/[A,A]A$. If $f$ is inner, then the image of $f$ is contained in $[A,A]$, and
the rest follows easily.
\end{proof}

Suppose now that $k$ has prime characteristic $p$.
If $A=kG$ for some finite group $G$, then the largest commutative quotient
of $kG$ is $kG/G'$, where $G'$ is the commutator subgroup of $G$.
Thus $[kG,kG]kG=I(kG')kG$. One can verify  this also directly by noting
the relation between additive and multiplicative commutators
$xyx^{-1}y^{-1}-1=$ $[x,y]x^{-1}y^{-1}$ for all $x$, $y\in G$. Thus Lemma
\ref{Der-Lemma1} specialises to the following observation.

\begin{Lemma} \label{Der-Lemma2}
Let $G$ be a finite group. Denote by $G'$ the commutator subgroup.
Every derivation on $kG$ induces a derivation on $kG/G'$, and every
inner derivation on $kG$ induces the zero map on $kG/G'$. 
Through this correspondence, the canonical surjection $G\to G/G'$induces a
Lie algebra homomorphism $\HH^1(kG)\to $ $\HH^1(kG/G')$.
\end{Lemma}

\begin{proof}
This is a special case of Lemma \ref{Der-Lemma1}, using the equality
$[kG,kG]kG=$ $I(kG')kG$ mentioned above.
\end{proof}

 We use without further 
comment the standard fact that for $P$ a finite $p$-group the augmentation ideal 
$I(kP)$ in $kP$ is equal to the Jacobson radical $J(kP)$. We denote by $\Phi(P)$ the
Frattini subgroup of $P$; this is the smallest normal subgroup of $P$ such that the
quotient $P/\Phi(P)$ is elementary abelian, with the convention $\Phi(P)=1$ if $P=1$.
The following is well-known; we sketch a proof for convenience.

\begin{Lemma} \label{PhiP-Lemma}
Let $P$ be a  finite $p$-group and $E$ a subgroup of $\Aut(P)$. 
The map sending $y\in P$ to $y-1\in J(\Fp P)$ induces an isomorphism
of $\Fp E$-modules
$$P/\Phi(P) \cong J(\Fp P)/J(\Fp P)^2.$$
\end{Lemma}

\begin{proof}
Set $J=J(\Fp P)$.
Let $x$, $y\in P$. Then $(x-1)(y-1)\in $ $J^2$. Since
$(x-1)(y-1) = (xy-1) - (x-1) - (y-1)$, it follows that $xy-1$ and $(x-1)+(y-1)$ have
the same image in $J/J^2$. Thus the map $x \mapsto x-1$ induces
a surjective group homomorphism $P\to J/J^2$. Since the
right side is an abelian group, the kernel of this group homomorphism 
contains the commutator subgroup of $P$, and
since $\Fp$ has characteristic $p$, the kernel contains also $x^p$ for all $x\in P$.
Thus the map $x\mapsto x-1$ yields a surjective group homomorphism
$P/\Phi(P)\to$ $J/J^2$. Both sides are easily seen to have the
same dimension, equal to the rank of the elementary abelian $p$-group $P/\Phi(P)$.
\end{proof}

Note that the
unit element of $P/\Phi(P)$ is mapped to the zero element in $J(\Fp P)/J(\Fp P)^2$
in the Lemma \ref{PhiP-Lemma}.  We will further need the following
observation regarding the hyperfocal subgroup $[P,E]$ of $P\rtimes E$ in $P$.

\begin{Lemma} \label{PhiP-Lemma-2}
Let $P$ be a  finite  $p$-group and $E$ a finite group  of order
prime to $p$ which acts on $P$. The following are equivalent.
\begin{itemize}
\item[{\rm (i)}] We have $[P,E]=P$.
\item[{\rm (ii)}] We have $[P/\Phi(P), E] = P/\Phi(P)$.
\item[{\rm (iii)}] The $\Fp E$-module $J(\Fp P)/J(\Fp P)^2$ has no nonzero  
trivial direct  summand.
\item[{\rm (iv)}] The $kE$-module $J(kP)/J(kP)^2$ has no nonzero trivial
direct summand.
\end{itemize}
\end{Lemma}

\begin{proof}
Clearly $[P/\Phi(P),E]$ is the image of $[P,E]$ under the canonical surjection 
$P\to P/\Phi(P)$, so (i) implies (ii) trivially. If $[P,E]$ is a proper subgroup of $P$,
then so is its image in $P/\Phi(P)$ since $\Phi(P)$ is the intersection of all maximal 
subgroups of $P$. Thus (ii) implies (i).  
By standard facts on coprime group actions, we have 
$P/\Phi(P)=[P/\Phi(P),E] \times C_{P/\Phi(P)}(E)$, thus (ii) is equivalent 
to the statement $C_{P/\Phi(P)}(E)=1$.
Under the isomorphism $P/\Phi(P)\cong$ $J(\Fp P)/J(\Fp P)^2$ from Lemma
\ref{PhiP-Lemma} this is equivalent to (iii).
Since $J(kP)=k\ten_{\Fp} J(\Fp P)$ and similarly for $J(kP)^2$, we have
$J(kP)/J(kP)^2 \cong $ $k\ten_\Fp J(\Fp P)/J(\Fp P)$ as $kE$-modules.
Setting $U=$ $J(\Fp P)/J(\Fp P)^2$, the equivalence of (iii) and (iv) follows from
the canonical  isomorphisms $k\ten_\Fp U^E\cong$ $k\ten_\Fp \Hom_{Fp E}(\Fp,U) \cong$
$\Hom_{kE}(k, k\ten_\Fp U) \cong$ $(k\ten_\Fp U)^E$, where for the second isomorphism
we make use of the well-known fact \cite[Corollary 1.12.11]{LiBookI} on scalar extensions
of homomorphism spaces.
\end{proof}

For convenience we draw attention to the following obvious fact.

\begin{Lemma}\label{Lemma-3}
Let $P$ be a finite $p$-group and $E$ a finite group of order prime to $p$ which
acts on $P$. Every element in $\HH^1(kP)^E$ has a representative in $\Der(kP)^E$.
\end{Lemma}

\begin{proof}
The canonical surjection $\Der(kP)\to \HH^1(kP)$ is $E$-stable, so 
remains surjective upon taking $E$-fixed points as $E$ is a $p'$-group.
\end{proof}

We will need the following fact from \cite{KL-Frobenius}.

\begin{Lemma}[{cf. \cite[Lemma 5.1]{KL-Frobenius}}]  \label{fpf-Lemma}
Let $P$ be a finite $p$-group and $\alpha$ an automorphism of $P$ with no 
nontrivial fixed point. Then the $kP$-module $(kP)_\alpha$, with $u\in P$ 
acting on $x\in P$  by $ux\alpha(u)^{-1}$, is projective.
\end{Lemma}

\begin{proof}
The hypothesis on $\alpha$ implies that if $u$ runs over all elements of $P$, then
so does $u\alpha(u)^{-1}$. Thus the given action of $P$ on itself is transitive 
because the $P$-orbit of $1$ is $P$. The Lemma follows.
\end{proof}

\section{The K\"unneth formula and solvability of $\HH^1$} \label{kuenneth-section} 

Let $k$ be a field. For  $A$ an associative unital $k$-algebra  and $m$ a positive 
integer  we denote as before by $\Der_m(A)$ the space of derivations on $A$ 
with image contained in $J(A)^m$. Given two algebras $A$, $B$, the
solvability of $\HH^1(A)$ and $\HH^1(B)$ does not necessarily imply the solvability 
of $\HH^1(A\tenk B)$.  The following observation implies that the slightly stronger 
condition from Corollary \ref{rad-Cor}   (iii) does extend to tensor products.

\begin{Proposition} \label{Kuenneth-prop}
Let $A$, $B$ be two associative unital $k$-algebras.  Suppose that the canonical 
maps $\Der_1(A)\to$ $\HH^1(A)$ and $\Der_1(B)\to$ $\HH^1(B)$  are surjective.
Then the map $\Der_1(A\tenk B)\to$ $\HH^1(A\tenk B)$ is surjective.
Suppose further that $[\Der_1(A),\Der_1(A)]\subseteq$ $\Der_2(A) + \IDer(A)$ and 
that  $[\Der_1(B), \Der_1(B)] \subseteq$ $\Der_2(B)+\IDer(B)$.  Then 
$$[\Der_1(A\tenk B), \Der_1(A\tenk B)] \subseteq \Der_2(A\tenk B) + \IDer(A\tenk B).$$
In particular, the Lie algebra $\HH^1(A\tenk B)$ is solvable.
\end{Proposition}

The proof of this Proposition is based on the  K\"unneth formula
\begin{Statement} \label{Kuenneth}
$$\HH^1(A\tenk B) \cong  
\ \mathrm{Z}(A)\tenk \HH^1(B)\ \oplus\ \HH^1(A)\tenk \mathrm{Z}(B),$$
\end{Statement}

\noindent 
where we use the canonical identifications  $\HH^0(A)\cong$  $\mathrm{Z}(A)$ 
and  $\HH^0(B)\cong$ $\mathrm{Z}(B)$. This formula extends in the obvious
way to tensor products of more than two algebras.
The K\"unneth  isomorphism \ref{Kuenneth}
 is induced by 
 with the map sending $z \ten g$ to the derivation
$$a\ten b \mapsto az \ten g(b)$$
on $A\tenk B$, where $a\in$ $A$, $b\in$ $ B$, $z\in \mathrm{Z}(A)$ and $g$ is a derivation on $B$,
together with the map sending $f \ten w$ to the derivation
$$a \ten b \mapsto f(a) \ten bw$$
on $A\tenk B$,  where $f$ is a derivation on $A$ and 
$w\in \mathrm{Z}(B)$.
A trivial verification shows that 
 if $g=$ $[d,-]$ for some $d\in B$  is an inner derivation on $B$,
then the derivation on $A\tenk B$ corresponding to $z\ten g$ is inner, equal to 
$[z\ten d, -]$.
Similarly, if $f=$ $[c,-]$ for some $c\in A$ is an inner derivation 
on $A$,  then the derivation on $A\tenk B$ corresponding to $f\ten w$ is inner, 
and equal to $[c\ten w,-]$.

The Lie bracket can be followed through the K\"unneth isomorphism as follows. 
given two derivations $g$, $g'$ on $B$ and $z$, $z'\in \mathrm{Z}(A)$, the 
Lie bracket of the derivations corresponding to  $z\ten g$, $z'\ten g'$ is  given by
\begin{Statement} \label{KuennethLie2}
$$[z\ten g, z'\ten g'] = zz' \ten [g,g'],$$
or explicitly, the right side is  the map
$$a \ten b \mapsto  azz' \ten [g,g'](b).$$
\end{Statement}
Similarly, given two derivations $f$, $f'$ on $A$ and $w$, $w'\in \mathrm{Z}(B)$, 
and  identifying $f\ten w$ with the derivation $a\ten b\mapsto f(a)\ten bw$,
the Lie  bracket of the derivations $f\ten w$, $f'\ten w'$ is given  by 
\begin{Statement} \label{KuennethLie1}
$$[f\ten w, f'\ten w'] = [f,f'] \ten ww',$$
or explicitly, the right side is  the map
$$a \ten b \mapsto [f,f'](a) \ten bww'. $$ 
\end{Statement}
The formulas \ref{KuennethLie2} and \ref{KuennethLie1}
show that the two summands in the K\"unneth decomposition \ref{Kuenneth} 
are both Lie subalgebras. Applied with $w=w'=1_B$ and $z=z'=1_A$, these 
formulas show that $\HH^1(A)$ and $\HH^1(B)$ are isomorphic to Lie subalgebras
of $\HH^1(A\tenk B)$, so if one of $\HH^1(A)$, $\HH^1(B)$ is not solvable, then
neither is $\HH^1(A\tenk B)$. Note though that the solvability of both $\HH^1(A)$,
$\HH^1(B)$ need not imply the solvability of $\HH^1(A\tenk B)$. 
By Lemma \ref{ZA-derivations} we have $f(z)\in$ $Z(A)$ and $g(w)\in$ $Z(B)$. We 
denote by $z\cdot f$ (resp. $w\cdot g$) the derivation on $A$ (resp. on $B$) given
by $(z\cdot f)(a)=$ $zf(a)$ (resp. $(w\cdot g)(b)=$ $wg(b)$).
The Lie bracket  $[f \ten w$, $z \ten g]$ is given by
\begin{Statement} \label{KuennethLie3}
$$[f\ten w, z\ten g] =    f(z) \ten w\cdot g - z\cdot f \ten g(w),$$ 
or equivalently, the right side is the map
$$a\ten b \mapsto f(az)\ten wg(b) - f(a)z\ten g(bw) =  af(z)\ten wg(b)  - zf(a) \ten bg(w).$$
In particular, we have 
$$[f \ten 1, z\ten g]= f(z) \ten g.$$ 
\end{Statement}

\noindent
Indeed, the first formula  uses the Leibniz rule applied to the terms $f(az)$ and $g(bw)$,
followed by cancelling two terms, and the last formula follows from applying
this to $w=1$ and using $g(1)=0$. 
The formula \ref{KuennethLie3} shows that the two summands in the K\"unneth formula
do not necessarily commute,
so this is not, in general, a direct product of Lie algebras. We note the following
consequence of this formula, used in the proof of Proposition \ref{Kuenneth-prop}.

\begin{Lemma} \label{KuennethRad}
Let $A$, $B$ be finite-dimensional associative unital $k$-algebras, 
and  $z\in \mathrm{Z}(A)$,  $w\in \mathrm{Z}(B)$, 
$f$ a derivation on $A$, and $g$ a derivation on $B$. 
Suppose that $\Im(f)\subseteq J(A)$ and $\Im(g) \subseteq J(B)$. Then 
$[f\ten w, z\ten g]$, regarded as a derivation on $A\tenk B$,
 has image contained in $J(A\tenk B)^2$. 
\end{Lemma}
 
\begin{proof}
The hypotheses and formula \ref{KuennethLie3} imply that the image of 
the derivation $[f\ten w, z\ten g]$ on $AA\tenk B$ is contained in 
$J(A)\ten J(B) \subseteq$  $(J(A)\ten B)(A\ten J(B))\subseteq$ $J(A\tenk B)^2$.
\end{proof}
 
\begin{proof}[Proof of Proposition \ref{Kuenneth-prop}]
Given a derivation $f$ on $A$ with image contained in $J(A)$ and an element 
$w\in \mathrm{Z}(B)$,  the derivation on $A\tenk B$ corresponding to $f\ten w$ 
has image contained in $J(A)\tenk B \subseteq$ $J(A\tenk B)$. Similarly, given a 
derivation $g$ on $B$ and $z\in \mathrm{Z}(A)$, the derivation on $A\tenk B$ 
corresponding to $z\ten g$ has image contained in $J(A\tenk B)$.
The K\"unneth formula \ref{Kuenneth} implies the first statement.
The second statement follows from combining
the formulas \ref{KuennethLie1}, \ref{KuennethLie2} and Lemma \ref{KuennethRad}.
The solvability of $\HH^1(A\tenk B)$ follows from Corollary \ref{rad-Cor}. 
\end{proof}

\begin{Lemma} \label{KuennethRad2}
Let $A$, $B$ be finite-dimensional associative unital $k$-algebras.
Suppose  that the canonical map $\Der_1(B)\to$ $\HH^1(B)$ is surjective.
Then the space
$$\ Z(A) \tenk \HH^1(B) \ \oplus\ \HH^1(A)\tenk J(Z(B)),$$ 
identified to its image in $\HH^1(A\tenk B)$, 
is a Lie ideal in  $\HH^1(A\tenk B)$. In particular, if $\HH^1(A)$ is non-zero, then
$\HH^1(A\tenk B)$ is not  a simple Lie algebra. 
\end{Lemma}

\begin{proof}
This follows from the formulas \ref{KuennethLie1} and \ref{KuennethLie2},
together with the fact that if $g$ is a derivation on $B$ with image contained
in $J(B)$, then, by Lemma \ref{ZA-derivations},   $g$ sends $Z(B)$ to 
$Z(B)\cap J(B)=$ $J(Z(B))$.  If $\HH^1(A)\neq 0$, then the space displayed in 
the statement does not contain  $\HH^1(A)\ten 1_B$,  so this is a proper ideal. 
\end{proof}
  
If the algebra $B$ is separable, then the K\"unneth formula yields an isomorphism
\begin{Statement}\label{Kuenneth-separable}
$$\HH^1(A\tenk B) \cong \HH^1(A) \tenk  \mathrm{Z}(B).$$
\end{Statement}
 
We will need one further special case of the K\"unneth formula for finite group 
algebras.   Given finite groups $G$, $H$, a $kG$-module $U$ and a $kH$-module 
$V$, we  have a natural graded $k$-linear isomorphism
\begin{Statement} \label{Kuenneth-groups}
$$\H^*(G\times H; U\tenk V) \cong \H^*(G; U) \tenk  \H^*(H; V),$$
\end{Statement}

\noindent where the grading on the right side is the total grading. 
Explicitly, for any positive integer $n$ we have
\begin{Statement} \label{Kuenneth-groups-n}
$$\H^n(G\times H; U\tenk V)\ \cong\ \oplus_{(i,j)}\ \H^i(G; U) \tenk \H^j(H; V),$$
\end{Statement} 

\noindent where $(i,j)$ runs over all pairs of non-negative integers such that $i+j=n$.
See for instance \cite[Theorem 3.5.6]{BenI}. For $n=1$, this yields an isomorphism
\begin{Statement} \label{Kuenneth-groups-1}
$$\H^1(G\times H; U\tenk V) \cong U^G\tenk \H^1(H; V)\ \oplus\ \H^1(G; U) \tenk V^G.$$
\end{Statement} 
Under this isomorphism an element in $U^G\tenk \H^1(H;V)$ given
by $u\ten \tau$ for some $u\in U^G$ and some $\tau\in Z^1(H; V)$  corresponds
to the element in $\H^1(G\times H; U\tenk V)$ given by the $1$-cocycle
$(x,y)\mapsto u \ten \tau(y)$. The analogous statement holds for elements
in the second summand.

\section{Calculations in twisted group algebras}  \label{twisted-Section} 

One of the standard tools for calculating the Hochschild cohomology of
a finite group algebra is the centraliser decomposition, which is shown in
\cite[Lemma 3.5]{Wi04} to carry over to crossed products, and in
particular therefore to twisted group algebras. We review very briefly what
we will need  in this paper; for more background material see for instance
\cite[Section 1.2]{LiBookI}.

Let $G$ be a finite group and let $k$ be a field. Let $\alpha\in Z^2(G; k^\times)$.
The twisted group algebra $k_\alpha G$ has a $k$-basis $\{\hat g\ |\ g\in G\}$
in bijection with the elements of $G$. The multiplication in $k_\alpha G$ is given by
$\hat g \hat h = $ $\alpha(g,h) \widehat{gh}$, for $g$, $h\in G$, extended bilinearly
to $k_\alpha G$.
The identity element in $k_\alpha G$ is $\alpha(1,1)^{-1}\hat 1$, and hence, for
$g\in G$, the inverse of $\hat g$ in $k_\alpha G$ is given by 
\begin{Statement}\label{inverse}
$$\hat g^{-1} = \alpha(1,1)^{-1}\alpha(g,g^{-1})^{-1} \widehat{g^{-1}},$$
\end{Statement}

\noindent where $g^{-1}$ is the inverse of $g$ in $G$. The isomorphism
class of $k_\alpha G$  depends only on the class of $\alpha$ in $H^2(G; k^\times)$,
and we may therefore assume that $\alpha$ is {\em normalised}; that is,
$\alpha(g,1)=1=\alpha(1,g)$ for all $g\in G$. This is
equivalent to requiring that $\hat 1$ remains the identity element in $k_\alpha G$.
We note that if $\alpha$ is normalised, then the inverse of $\hat g$ in $k_\alpha G$
is equal to $\hat g^{-1}=$ $\alpha(g,g^{-1})^{-1} \widehat{g^{-1}}$. 
A short calculation shows that 
the conjugation action in $k_\alpha G$ is given by
\begin{Statement} \label{conj-action-1}
$${^{\hat g}{\hat h}} = \hat g\hat h \hat g^{-1} = \lambda(g,h)  \widehat{{^gh}},$$
\end{Statement}

\noindent where $g$, $h\in G$ and where $\lambda(g,h)\in$ $k^\times$ is given by 
the formula
$$\lambda(g,h) = 
\alpha(g,h)\alpha(gh, g^{-1})\alpha(g,g^{-1})^{-1}\alpha(1,1)^{-1}.$$ 
In particular, we have $\lambda(1,h)=1=\lambda(g,1)$.
If $N$ is a normal subgroup of $G$ and $\alpha\in Z^2(G/N; k^\times)$ inflated to
$G$ via the canonical surjection, still denoted by $\alpha$, 
and if we assume in addition that $\alpha$
is normalised, then for $g\in N$ and $h\in G$ the above formula yields
$\lambda(g,h)=1$, hence in that case we have
\begin{Statement} \label{conj-action-2}
$$^{\hat g}{\hat h} =  \widehat{{^gh}}.$$
\end{Statement}
For $M$ a $k_\alpha G$-$k_\alpha G$-bimodule, we have a standard adjunction 
isomorphism
\begin{Statement}\label{HH-adjunction}
$$\HH^*(k_\alpha G; M) \cong \H^*(G; M),$$
\end{Statement}

\noindent where $g\in G$ acts on $m\in M$ by ${^gm}=\hat g m\hat g^{-1}$, 
having checked that this is well-defined. Note that while $M$ is considered as a 
$kG$-module, the cohomology $\H^*(G; M)$ still depends on $\alpha$, even 
though $\alpha$ does not explicitly appear in the notation. In particular, with 
$M=k_\alpha G$,   the group $G$ acts on $k_\alpha G$  with $g\in G$ acting
by conjugation with $\hat g$, and we have a graded isomorphism 
$\HH^*(k_\alpha G)\cong$ $\H^*(G; k_\alpha G)$, which is the first step towards
the  centraliser decomposition
of $\HH^*(k_\alpha G)$ in the proof of \cite[Lemma 3.5]{Wi04}. We will
need the isomorphism \ref{HH-adjunction} 
 in degree $1$, where this is given explicitly as follows.

\begin{Lemma} \label{twisted1}
Let $G$ be a finite group, $k$ a field, and $\alpha\in$ $Z^2(G; k^\times)$.
Let $M$ be a $k_\alpha G$-$k_\alpha G$-bimodule.
Let $d : k_\alpha G\to M$ be a $k$-linear map and $\tau : G\to M$ a map
such that
$d(\hat g)= \tau(g)\hat g$ 
for all $g\in G$. 
Then $d$ is a derivation if and only if $\tau$ is a $1$-cocycle. Moreover,  the 
correspondence $\tau\mapsto d$ induces an isomorphism 
$\H^1(G;  M)\cong$ $\HH^1(k_\alpha G; M)$.
\end{Lemma}

\begin{proof}
Let $g$, $h\in G$. We have
$$d(\hat g\hat h)= \alpha(g,h) d(\widehat{gh})=\alpha(g,h)\tau(gh) \widehat{gh}
= \tau(gh) \hat g\hat h$$
and
$$d(\hat g)\hat h + \hat g d(\hat h) = \tau(g)\hat g \hat h + \hat g \tau(h) \hat h =
(\tau(g) + {^g{\tau(h)}}) \hat g\hat h.$$
Thus $d$ is a derivation if and only if $\tau$ is a $1$-cocycle.
We have $\tau(g) = m-{^gm}$  for some $m\in M$ if and only if
$d(\hat g)=$ $\tau(g)\hat g =$ $m\hat g-{^gm} \hat g=$ $m\hat g-\hat g m=$ $[m,\hat g]$.
Thus  $d$ is an inner derivation if and only if $\tau$ is a $1$-coboundary.
The result follows. 
\end{proof}

By standard facts on group cohomology, this Lemma implies that if $N$ is
a normal subgroup of $G$ of index invertible in $k$, then, using
\cite[Proposition III.10.4]{Brown} and the isomorphism \ref{HH-adjunction}
with $N$ instead  of $G$,  we have an isomorphism
\begin{Statement} \label{HH-normal}
$$\HH^*(k_\alpha G; M) \cong \HH^*(k_\alpha N; M)^{G/N} \cong \H^*(N; M)^{G/N}, $$
\end{Statement}

\noindent where we use the same letter $\alpha$ for the restriction of $\alpha$ to 
$N\times N$.
The action of $G/N$ on the last two terms is induced by the conjugation action of $G$
on $k_\alpha N$, on $N$, and on $M$, where we note that
$N$ acts as identity on $\HH^*(k_\alpha N; M)$. (This is
just the version for twisted group algebra of the arguments in the proof
of \cite[Theorem 3.2]{Murphy23}.) If $M=k_\alpha G$ and if $\alpha$ is
the inflation to $G\times G$ of a $2$-cocycle in $Z^2(G/N; k^\times)$,
then $k_\alpha G= kN \oplus k_\alpha(G\setminus N)$, and hence, still
assuming that the index of $N$ in $G$ is invertible in $k$, the first isomorphism in
\ref{HH-normal} specialises to
\begin{Statement}\label{HH-normal-2}
$$\HH^*(k_\alpha G) \cong \HH^*(kN)^{G/N} \oplus 
\HH^*(kN; k_\alpha(G\setminus N)^{G/N},$$
\end{Statement} 

\noindent which shows in particular that $\HH^1(kN)^{G/N}$ is a Lie subalgebra of
$\HH^1(k_\alpha G)$.

\medskip
In what follows we will frequently identify the elements in $G$ to their images
in $k_\alpha G$. In that case, for two elements $g$, $h\in G$, we will denote by 
$gh$ the product in the group and by $g\cdot h$ the product in $k_\alpha G$.

\section{Derivations  on $k_\alpha(P\rtimes E)$ and $E$-stable derivations on
$kP$}  \label{E-stable-section} 

Let $p$ be a prime and $k$ a field of characteristic $p$.
We will apply the above calculations in twisted group algebras to groups of
the form $P\rtimes E$ for some finite $p$-group $P$, some $p'$-subgroup $E$ 
of $\Aut(P)$, and some  $\alpha\in Z^2(E; k^\times)$ inflated  to $P\rtimes E$ via 
the canonical surjection $P\rtimes E\to E$. The resulting $2$-cocycle in 
$Z^2(P\rtimes E; k^\times)$ will abusively again be denoted by the same letter
$\alpha$.  That is, for $u$, $v\in P$ and $x$, $y\in E$ we  have 
$$\alpha(ux, vy)= \alpha(x,y). $$ 
If we assume in addition that $\alpha$ is normalised, then $\alpha(x,y)$ 
is equal to $1$ if one of $x$, $y$ is trivial. Note that this implies in 
particular that $kP$ is a subalgebra of $k_\alpha(P\rtimes E)$ and that
$k_\alpha(P\rtimes E)$ is isomorphic to $k(P\rtimes E)$ as a $kP$-$kP$-bimodule
(cf. \cite[Corollary 5.3.8]{LiBookI}).
The conjugation action of $E$ on $kP$ and on $k_\alpha(P\rtimes E)$
induces an action of $E$ on $\HH^*(kP; k_\alpha(P\rtimes E))$.

\begin{Lemma} \label{twisted2}
Let $P$ be a finite $p$-group and $E$ a $p'$-subgroup of $\Aut(P)$. 
Let $\alpha\in Z^2(E; k^\times)$ inflated to $P\rtimes E$ via the
canonical surjection $P\rtimes E\to E$. 
\begin{itemize}
\item[{\rm (i)}] We have canonical graded isomorphisms
\begin{align*}
\HH^*(k_\alpha(P\rtimes E)) & \cong  (\HH^*(kP; k_\alpha(P\rtimes E))^E \\
 & \cong  \HH^*(kP)^E \ \oplus\  (\oplus_{e\in E\setminus \{1\}}\ \HH^*(kP; kP\cdot e))^E.
 \end{align*}
 \item[{\rm (ii)}] 
If $E$ is abelian, then $E$ stabilises every summand in the last direct sum in (i), and
we have canonical graded isomorphisms
$$\HH^*(k_\alpha (P\rtimes E)) \cong\ \oplus_{e\in E}\ \HH^*(kP; kP\cdot e)^E\ 
\cong\  \oplus_{e\in E}\ \H^*(P; kP\cdot e)^E.$$
\item[{\rm (iii)}] 
If  $E$ acts freely on $P\setminus \{1\}$, then for all positive 
integers $n$ we have
$$\HH^n(k_\alpha(P\rtimes E))\cong \HH^n(kP)^E.$$
For $n=1$ this is a Lie algebra isomorphism.
\end{itemize}
\end{Lemma}

\begin{proof}
Since $E$ is a $p'$-group, the statements (i) and (ii) follow from the isomorphism
\ref{HH-normal}.
Statement (iii)  is well-known (see e.g. \cite[Proposition 5.2]{KL-Frobenius}, 
or \cite[Theorem 3.2]{Murphy23}) and  follows from (i) and together
with the  fact that $H^n(P; kP\cdot e)=0$ for $e\neq 1$ by Lemma \ref{fpf-Lemma}, 
for any positive integer $n$.
\end{proof}

For the sake of completeness, we show that  the Lie algebra
embedding  of $\HH^1(kP)^E$ into $\HH^1(k_\alpha(P\rtimes E))$ 
to  a canonical map at the level of derivations.

\begin{Proposition} \label{twisted-PE-prop}
Let $P$ be a finite $p$-group and $E$ a finite $p'$-group acting on $P$. 
Let $\alpha\in$ $Z^2(E; k^\times)$, inflated to $P\rtimes E$ via the canonical 
surjection $P\rtimes E\to E$.
\begin{itemize}
\item[{\rm (i)}]
Every $E$-stable derivation $f$ on $kP$ extends uniquely to a derivation
$\hat f$ on $k_\alpha(P\rtimes E)$ with $k_\alpha E\subseteq$ $\ker(\hat f)$.
\item[{\rm (ii)}]
The correspondence $f\mapsto \hat f$ induces an injective Lie algebra
homomorphism $\HH^1(kP)^E\to $ $\HH^1(k_\alpha(P\rtimes E))$. 
\item[{\rm (iii)}]
If $E$ acts freely on $P\setminus \{1\}$, then  the Lie algebra homomorphism
 in {\rm{(ii)}}  is an isomorphism.
\end{itemize}
\end{Proposition}

\begin{proof}
We may assume that $\alpha$ is normalised; that is, $\alpha(1, x)=$ $1=$ 
$\alpha(x,1)$ for $x\in P\rtimes E$.  Since $\alpha$ is inflated to $P\rtimes E$ 
via the canonical surjection it follows that for $u\in P$ and $y\in E$ we have 
$\alpha(u,y)=$ $1=$ $\alpha(y,u)$. Equivalently,  the image in $k_\alpha(P\rtimes E)$
of the product $uy$ (resp. $yu$)  in $P\rtimes E$ is equal to the product 
$u\cdot y$ (resp. $y\cdot u$)  in $k_\alpha(P\rtimes E)$.

 Let $f\in \Der(kP)^E$.
Define a linear map $\hat f$ on $k_\alpha(P\rtimes E)$  by setting 
$$\hat f(uy)= f(u) \cdot y,$$ 
where $uy$ is the product in $P\rtimes E$ and where
the right side is the product taken in $k_\alpha(P\rtimes E)$. This defines $\hat f$
uniquely as a linear map on $k_\alpha(P\rtimes E)$ which extends $f$ and vanishes 
on  $k_\alpha E$. The Leibniz rule implies that if there is a derivation on 
$k_\alpha(P\rtimes E)$  which extends $f$ and which vanishes on $k_\alpha E$, then 
it must be equal to $\hat f$.
It remains  to check that $\hat f$ is indeed a derivation. 

Let $u$, $v\in P$ and $y$, $z\in E$. Calculating in $k_\alpha(P\rtimes E)$ and using
that $\alpha$ is normalised and inflated to $P\rtimes E$, we have 
$$(uy)\cdot (vz)= uyvz \alpha(uy, vz)= u(^yv)yz \alpha(y,z) = 
u(^yv)\cdot y \cdot z.$$
Thus
$$\hat f((uv)\cdot (yz)) = \hat f(u(^yv) \cdot y\cdot z) =    f(u(^yv)) \cdot y\cdot z.$$
We need to show that this is equal to $\hat f(uy) \cdot vz + uy \cdot \hat f(vz)$.
Using that $f$ is $E$-stable as well as a derivation, together with the comments 
preceding this  Proposition, we have
$$\hat f(uy) \cdot vz + uy \cdot \hat f(vz) = f(u)\cdot y \cdot vz + uy \cdot f(v)\cdot z =
f(u)(^yv) \cdot y\cdot z + u( {^yf(v)}) \cdot y\cdot z = $$
$$f(u)(^yv) \cdot y\cdot z + u f(^yv) \cdot y\cdot z = f(u(^yv))\cdot y\cdot z,$$
which implies that $\hat f$ is a derivation on $k_\alpha(P\rtimes E)$. 
The construction of $\hat f$ implies that the assignment $f\mapsto \hat f$ is a Lie algebra
homomorphism $\Der(kP)^E\to$ $\Der(k_\alpha(P\rtimes E))$. If $f$ is inner, hence equal
to $[c,-]$ for some $c\in kP$, then $\hat f(uy)=$ $f(u)y=[c,u] y$. The $E$-stablility of $f$
implies that  $[c,{^yu}]=$ $[^yc,{^yu}]$. This holds for all $u\in P$, and hence $[c,-]$ and
$[^yc,-]$ have the same restriction  to $kP$, which is equal to $f$.
Thus we may replace $c$ by $\frac{1}{|E|}\Tr_1^E(c)$,
and then $\hat f=[c,-]$, showing that $\hat f$ is an inner derivation. 
Conversely, if 
$\hat f$ is an inner derivation, then $\hat f=[d,-]$ for some $d\in k_\alpha(P\rtimes E)$
which centralises $k_\alpha E$. Writing $d=\sum_{e\in E} c_e e$ for some $c_e\in kP$,
one sees that $f=[c_1,-]$, so $f$ is inner. This, together with Lemma \ref{Lemma-3},
 shows that the assignment
$f\mapsto \hat f$ induces an injective Lie algebra homomorphism $\HH^1(kP)^E\to$
$\HH^1(k_\alpha(P\rtimes E))$. The last statement follows from Lemma
\ref{twisted2} (iii). 
\end{proof}

\begin{Remark}  \label{frobenius-2-cocycle}
With the notation of Proposition \ref{twisted-PE-prop}, if $E$ acts freely on 
$P\setminus \{1\}$, then $P\rtimes E$ is a Frobenius group. The structural properties 
of Frobenius groups, as described in \cite[Theorem 10.3.1]{Gor}, imply that if $k$ is 
algebraically closed, then $H^2(E;k^\times)$  is trivial.   Thus  $\alpha$ may be chosen 
to be $1$ in that case.
\end{Remark}

\begin{Lemma} \label{lem1}
Let $P$ be a   finite $p$-group and $E$ a finite $p'$-group acting on $P$.
Set $Q=\Phi(P)$. 

\begin{itemize}
\item[{\rm (i)}]
The canonical surjection $P\to P/Q$ induces a Lie algebra homomorphism
$\HH^1(kP) \to$ $ \HH^1(kP/Q)$.

\item[{\rm (ii)}]
If $P$ is abelian, then  $J(kQ)\subseteq$  $J(kP)^p$, and  the canonical surjection 
$P\to P/Q$ induces a surjective Lie algebra homomorphism $\HH^1(kP) \to$
$\HH^1(kP/Q)$ with nilpotent kernel.

\item[{\rm (iii)}]
If $P$ is abelian, then the Lie algebra homomorphism from {\rm (ii)} induces a 
surjective Lie algebra homomorphism $\HH^1(kP)^E\to \HH^1(kP/Q)^E$ 
with nilpotent kernel.
\end{itemize}
\end{Lemma}

\begin{proof}
Since $Q$ contains the commutator subgroup $P'$ of $P$, 
the algebra homomorphism $kP\to kP/Q$ factors through
the algebra homomorphism $kP\to kP/P'$. By Lemma \ref{Der-Lemma2},
this homomorphism induces a Lie algebra homomorphism $\HH^1(kP)\to$
$\HH^1(kP/P')$. Thus we may assume that $P$ is abelian. 
The kernel of the canonical algebra homomorphism $kP\to kP/Q$ is
equal to $J(kQ)kP$.
Since $P$ is abelian, the subgroup $Q$ consists of all elements of the for $x^p$,
with $x\in P$. Thus $J(kQ)$ is spanned by the set of elements of the form
$x^p-1=$ $(x-1)^p$. In particular, $J(kQ)\subseteq$ $J(kP)^p$, which is the first
statement in (ii). 
Since $P$ is abelian, any derivation $f$ on $kP$ 
satisfies $f((x-1)^p)=p(x-1)^{p-1}f(x) =0$. Thus the kernel of $f$ contains
$J(kQ)$. The Leibniz rule implies that $f$ preserves the ideal $J(kQ)kP$.
Thus $f$ induces a derivation on $kP/Q$. This shows (i).

For the surjectivity statement in (ii), one can either play this
back to the case where $P$ is cyclic via the K\"unneth formula, and then
show by direct verification that every derivation on $kP/Q$ is induced
by a derivation on $kP$. Or one can use the formula $\HH^1(kP)\cong$
$kP\tenk \H^1(P,k)$ from \cite{Holma}, with the analogous formula for $P/Q$ 
and the
fact that $\H^1(P; k)=$ $\H^1(P/Q; k)$, since $\Hom(P, k)$ can be identified
with $\Hom(P/Q,k)$.

It remains to show in (ii)  that the kernel of the map $\HH^1(kP)\to$ $\HH^1(kP/Q)$
is nilpotent. Let $f$ be a derivation on $kP$ inducing the zero map on $kP/Q$.
Then the image of $f$ is contained in $J(kQ)kP$, and by the above this is contained
in $J(kP)^p$. The result follows from Proposition \ref{rad-Prop} (applied
with $C=k\cdot 1$, which is in the kernel of every derivation on $kP$). 
This shows (ii).

Since $E$ is a $p'$-group, statement (iii) is an immediate consequence of (ii).
\end{proof}

\begin{Proposition} \label{Pe-prop}
Let $P$ be a  finite $p$-group and $E$ a finite  $p'$-group acting on $P$.
Suppose that $[P,E]=P$. Let $e\in Z(E)$. 
Then $kPe$ is an $E$-stable $kP$-$kP$-bimodule summand of 
$k(P\rtimes E)$, and for every derivation $f : kP\to kPe$ which
is $E$-stable we have $\Im(f)\subseteq J(kP) e$.
\end{Proposition}

\begin{proof}
Set $J=J(kP)$. Clearly $kPe$ is a $kP$-$kP$-bimodule summand of
$k(P\rtimes E)$, and it is $E$-stable because $e\in$ $Z(E)$.
Since $f$ is a derivation, we have $f(J^2)\subseteq Je$, which in turn is
in the kernel of the augmentation map $kPe \to k$. 
Since $J$ is $E$-stable, it follows that the restriction of $f$
to $J$ is an $E$-stable map, as is the composition with the
augmentation map $\epsilon : kP e\to k$.  Thus the
 map $\epsilon\circ f|_J  : J \to k$ sends $J^2$ to zero,
 hence factors through the canonical surjection $J\to$ $J/J^2$.
We therefore get a commutative diagram of $kE$-modules 
 $$\xymatrix{ J \ar[rr]^{f|_J} \ar[d] & & kPe \ar[d]^{\epsilon} \\
 J/J^2 \ar[rr]_g & & k}$$
Since $[P,E]=P$,  it follows from Lemma \ref{PhiP-Lemma-2} that $J/J^2$ has 
 no nonzero trivial  summand as a $kE$-module. Since $J/J^2$ is also semi-simple 
as a $kE$-module this implies that $g$ is zero, hence that $f$ sends $kP$ to $Je$. 
\end{proof}

\begin{Corollary} \label{Pe-cor-1}
Let $P$ be a  finite $p$-group and $E$ a finite $p'$-group acting on $P$.
Suppose that $[P,E]=P$.
For every $E$-stable derivation $f : kP\to kP$ we have $\Im(f)\subseteq J(kP)$.
\end{Corollary}

\begin{proof}
This follows from Proposition \ref{Pe-prop} applied to $e=1$.
\end{proof}

Statement (ii) in the following Corollary is the special case of Theorem
\ref{twisted-PE-thm} (i) with $\alpha$ the trivial $2$-cocycle. 

\begin{Corollary} \label{Pe-cor-2}
Let $P$ be a  finite $p$-group and $E$ a finite abelian $p'$-group 
acting on $P$. Suppose that $[P,E]=P$. 
\begin{itemize}
\item[{\rm (i)}]
For every $E$-stable derivation $f : kP\to k(P\rtimes E)$
 we have $\Im(f)\subseteq J(k(P\rtimes E))$.

\item[{\rm (ii)}]
Every class in $\HH^1(k(P\rtimes E))$ is represented by a derivation 
$f : k(P\rtimes E)\to$ $k(P\rtimes E)$ which vanishes on $kE$, and any such 
derivation  $f$ satisfies  $\Im(f)\subseteq$ $J(k(P\rtimes E))$.
\end{itemize}
\end{Corollary}
 
\begin{proof}
Since $E$ is a $p'$-group, the algebra $kE$ is separable, and hence 
we have $J(k(P\rtimes E))=$ $J(kP) \cdot k(P\rtimes E))$.
Since $E$ is abelian, by Lemma \ref{twisted2},  $f$ is the sum of $E$-stable
derivations $kP\to kPe$, with $e\in E$, and hence (i) follows from 
Proposition \ref{Pe-prop}. We have
$k(P\rtimes E)=$ $kE \oplus J(k(P\rtimes E))$, and thus,
 by Lemma \ref{Czero}, every class in $\HH^1(k(P\rtimes E))$
 is represented by a derivation $f$ which vanishes on $kE$. The Leibniz 
 rule implies that  $f$ is then a $kE$-$kE$-bimodule endomorphism 
 of $k(P\rtimes E)$. 
 Thus $f$ is determined  by its restriction  $f|_{kP} : kP \to k(P\rtimes E)$. 
 It follows from (i) that $f$ sends $kP$ to 
 $J(k(P\rtimes E))$.
 Since $f$ is in particular a right $kE$-homomorphism, $f$ sends $kPe$ to 
 $J(k(P\rtimes E))$ for all $e\in E$, whence (ii).
 \end{proof}

\begin{Lemma} \label{lem3}
Let $P$ be a non-trivial finite abelian $p$-group, and set $J=J(kP)$. Denote by 
$\Der_1(kP)$ the Lie subalgebra of $\Der(kP)=$ $\HH^1(kP)$ of derivations 
which preserve $J$. Set $n=\dim_k(J/J^2)$. The canonical map
$$\Der_1(kP) \to \End_k(J/J^2)\cong \gl_n(k)$$
is a surjective Lie algebra homomorphism. The kernel of this homomorphism
is the nilpotent ideal $\Der_2(kP)$ of derivations with image contained in $J^2$.
\end{Lemma}

\begin{proof}
By Lemma \ref{lem1} we may assume that $P$ is elementary abelian.
Since $n$ is the rank of $P$ (cf. Lemma \ref{PhiP-Lemma}), we may write
$P=\prod_{i=1}^n \langle x_i\rangle$. For any two $i$, $j$ such that $1\leq i, j\leq n$
there is a unique derivation $d_{i,j}$ on $kP$ which sends $x_i-1$ to $x_j-1$ and
$x_{i'}-1$ to $0$, where $i'\neq i$, $1\leq i'\leq n$. Since the image of the set
$\{x_i-1\}_{1\leq i\leq n}$ in $J/J^2$ is a $k$-basis, the first statement follows,
and the second statement follows from Proposition \ref{rad-Prop}.
\end{proof}

\begin{Lemma} \label{lem4}
Let $P$ be a non-trivial finite abelian $p$-group, and set $J=J(kP)$. Let $E$ be 
a finite $p'$-group acting on $P$ such that $[P,E]=P$. Then
$\Der(kP)^E=$ $\Der_1(kP)^E$. Write
$$J/J^2 \cong \oplus_{j=1}^r\ S_j^{\oplus n_j}$$
with pairwise non-isomorphic simple $kE$-modules $S_j$ and positive integers 
$n_j$ ($1\leq j \leq r$). Then $\F_j = \End_{kE}(S_j)$ is a finite-dimensional commutative 
extension field of $k$, where $1\leq j\leq r$, and the  canonical map
$\Der_1(kP)\to \End_k(J/J^2)$ induces a surjective Lie algebra homomorphism
$$\Der(kP)^E \to \End_{kE}(J/J^2)\cong \prod_{j=1}^r \gl_{n_j}(\F_j).$$
The kernel of this Lie algebra homomorphism is the nilpotent
ideal $\Der_2(kP)^E$ of $E$-stable derivations on $kP$  with image in $J^2$.
\end{Lemma}

\begin{proof}
By Lemma \ref{Pe-cor-1} we have $\Der(kP)^E = $ $\Der_1(kP)^E$. The surjective
map $\Der_1(kP)\to$ $\End_k(J/J^2)$ from Lemma \ref{lem3} remains
surjective upon taking $E$-fixed points since $E$ has order prime to $p$.
Thus this map induces a surjective Lie algebra homomorphism
$\Der(kP)^E\to$ $\End_{kE}(J/J^2)$ with a nilpotent kernel as stated.
The rest follows from decomposing the semisimple $kE$-module $J/J^2$
as a direct sum of its isotypic components.
\end{proof}

\begin{Lemma} \label{lem5}
Let $P$ be a non-trivial  finite $p$-group and $E$ a finite $p'$-group acting on $P$.
The $kE$-module $J(kP)/J(kP)^2$ is multiplicity free if and only if the
$\Fp E$-module $P/\Phi(P)$ is multiplicity free.
\end{Lemma}

\begin{proof}
Since  $J(kP)/J(kP)^2 \cong k\ten_{\Fp} J(\Fp P)/J(\Fp P)^2$, it
follows from standard properties of coefficient extensions 
(e. g. \cite[Theorem 9.21.(b)]{Isaacs}, or \cite[Ch. I, Theorem 19.4]{Feitbook}, 
or \cite[(30.33)]{CRI}) 
 that the $\Fp E$-module $J(\Fp P)/(J(\Fp P)^2$ is multiplicity free if and only if
the $kE$-module $J(kP)/J(kP)^2$ is multiplicity free. Thus Lemma
\ref{PhiP-Lemma} implies the result.
\end{proof}

\section{Proof of Theorem \ref{frobenius-block-thm} and 
Theorem \ref{E-stable-derivations-thm}} 
\label{Proof-Section} 

\begin{proof}[Proof of Theorem \ref{E-stable-derivations-thm}]
Let $P$ be a finite abelian $p$-group and $E$ a finite $p'$-group acting on $P$ 
such that $[P,E]=P$.
Set $J=J(kP)$. 
By Lemma \ref{lem5},  $P/\Phi(P)$ is multiplicity free as an $\Fp E$-module if and only if
 $J/J^2$ is multiplicity free as a $kE$-module.
Since $P$ is abelian, we have
$\HH^1(kP)=$ $\Der(kP)$, hence $\HH^1(kP)^E=$ $\Der(kP)^E$. By Lemma \ref{Pe-cor-1} 
all $E$-stable derivations on $kP$ have image in $J(kP)$; that is, $\Der(kP)^E=$ 
$\Der_1(kP)^E$,  where the notation is as in Lemma \ref{lem4}. It follows from Lemma 
\ref{lem4} that $\HH^1(kP)^E$ is solvable if and only if $\Der_1(kP)^E/\Der_2(kP)^E$ 
is a solvable Lie algebra. 

If  $J/J^2$ is multiplicity free as a $kE$-module, then Lemma \ref{lem4} 
implies that  
$$\Der_1(kP)^E/\Der_2(kP)^E\cong \prod_{j=1}^r \gl_1(\F_j)$$ 
for some  commutative extension fields $\F_j$ of $k$,  and hence this Lie algebra 
is abelian,  which implies that  $\HH^1(kP)^E$ is solvable.

If $J/J^2$ is not multiplicity free, then $\Der_1(kP)^E/\Der_2(kP)^E$ has a direct factor
isomorphic to $\gl_n(\F)$ for some extension field $\F$ of $k$ and some integer $n\geq 2$.
Thus if $p$ is odd, then $\gl_n(\F)$ is not solvable, hence neither is $\HH^1(kP)^E$.
This completes the proof of Theorem \ref{E-stable-derivations-thm}.
\end{proof}

In order to complete the proof of Theorem \ref{frobenius-block-thm}, we
summarise the results in the case of a free $p'$-action on an abelian $p$-group.

\begin{Theorem} \label{frobenius-2-thm}
Let $P$ be a non-trivial finite abelian $p$-group and $E$ a $p'$-subgroup of 
$\Aut(P)$  acting freely on $P\setminus \{1\}$. Then the following hold.

\begin{itemize}
\item[{\rm (i)}] 
We have $\HH^1(k(P\rtimes E))\cong$  $\HH^1(kP)^E$ as Lie algebras.
\item[{\rm (ii)}] 
Suppose $E$ is non-trivial. 
Every class in $\HH^1(k(P\rtimes E))$ is represented by a derivation on 
$k(P\rtimes E)$ with image contained in the Jacobson radical $J(k(P\rtimes E))$. 
\item[{\rm (iii)}]
Suppose $E$ is non-trivial. If the  $\Fp E$-module $P/\Phi(P)$ is multiplicity free, 
then the Lie algebra $\HH^1(k(P\rtimes E))$ is solvable. The converse holds if $p$ 
is odd.
\end{itemize}
\end{Theorem} 

\begin{proof}
The first statement follows from Lemma \ref{twisted2} (iii) or Proposition
\ref{twisted-PE-prop}. The hypothesis $E\neq 1$ in (ii) and (iii) together with
the free action of $E$ on $P\setminus \{1\}$ imply that $[P,E]=P$. 
The remaining statements  follow from Theorem \ref{E-stable-derivations-thm}.
\end{proof}

\begin{proof}[Proof of Theorem \ref{frobenius-block-thm}]
With the notation and hypotheses of Theorem \ref{frobenius-block-thm}, by a result 
of  Puig \cite[6.8]{Puabelien} (see also \cite[Theorem 10.5.1]{LiBookII}) there is a 
stable equivalence of Morita type between $B$ and $k(P\rtimes E)$. By 
\cite[Theorem 10.7]{KLZ} this implies that there is a Lie algebra isomorphism 
$\HH^1(B)\cong$ $\HH^1(k(P\rtimes E))$. Thus Theorem \ref{frobenius-block-thm}
follows from Theorem \ref{frobenius-2-thm}.
\end{proof}

The following observation is a combination of some of the above results in 
conjunction with the structure theory of blocks with a normal defect group, slightly
generalising \cite[Proposition 5.2]{KL-Frobenius}, \cite[Theorem 3.2]{Murphy23}. 
We state this here for context  and for future reference.

\begin{Theorem} \label{normal-defect-thm} 
Let $G$ be a finite group and $B$ a  block with a non-trivial normal defect group 
$P$ and inertial quotient $E$. Suppose that $k$ is large enough for $B$.
Then $\HH^1(kP)^E$ is canonically isomorphic to a Lie subalgebra of $\HH^1(B)$.
If in addition  $E$ acts freely on $P\setminus \{1\}$, then 
$\HH^1(kP)^E\cong$ $\HH^1(B)$.
\end{Theorem}

\begin{proof}
Let $B$ be a block of a finite group algebra $kG$ with a non-trivial defect group 
$P$ which is  normal in $G$. Suppose that $k$ is large enough for $B$. Then, by 
\cite[Theorem A]{Kuenormal}
(see \cite[Theorem 6.14.1]{LiBookII} for an exposition of related material) 
the block $B$ is Morita equivalent to a twisted group algebra of the form 
$k_\alpha(P\rtimes E)$,  where $E$ is a $p'$-subgroup of $\Aut(P)$ and $\alpha\in$
$Z^2(E,k^\times)$, inflated to $P\rtimes E$ via the canonical surjection 
$P\rtimes E\to E$. Thus $\HH^1(B)\cong$  $\HH^1(k_\alpha(P\rtimes E))$ as Lie 
algebras. Theorem \ref{normal-defect-thm} follows from Lemma
\ref{twisted2} or Proposition \ref{twisted-PE-prop}.
\end{proof}

\begin{Remark}
Since the Lie algebra $\HH^1(B)$ of a block $B$ of a finite group algebra 
$kG$ is preserved under stable equivalences of Morita type 
(cf. \cite[Theorem 10.7]{KLZ}),  it follows that the conclusion of Theorem 
\ref{normal-defect-thm} holds  if there is a stable  equivalence of Morita 
type between $B$ and its  Brauer correspondent;  this includes inertial blocks
(these are blocks which are Morita equivalent to their Brauer correspondents
via bimodules with endopermutation source; cf. \cite[2.16]{Puig11}).
\end{Remark}

\section{Proof of Theorem \ref{twisted-PE-thm} and related results} 
\label{proof-section-2}

\begin{proof}[Proof of Theorem \ref{twisted-PE-thm}]
Let $P$ be a non-trivial finite abelian $p$-group and $E$ an abelian $p'$-subgroup of 
$\Aut(P)$ with  $[P,E]=P$. 
Let $\alpha\in Z^2(E;k^\times)$ inflated to $P\rtimes E$ via the canonical
surjection $P\rtimes E\to$ $E$. We may assume that $\alpha$ is normalised.

Note that the condition $[P,E]=P$  forces $|P|\geq 3$. 
Since $E$ is an abelian $p'$-group, by  Lemma  \ref{twisted2}, we have a
canonical isomorphism
\[\HH^1(k_\alpha (P\rtimes E))\cong
\HH^1(kP)^E\oplus \bigoplus_{e\in E\setminus\{1\}} \H^1(P;kP\cdot e)^E.
\]
Again by Lemma \ref{twisted2},  for each $f\in \H^1(P; kP\cdot e)^E$, the 
corresponding element in 
$\HH^1(k_\alpha (P\rtimes E))$ under this isomorphism is represented by 
$d_f\in \Der(k_\alpha(P\rtimes E))$ such that for each $g\in P$, $d_f(g)=f(g)\cdot g$ 
and  $d_f$ vanishes on $E$.

Fix a nonidentity element $e\in E$. Set $T=C_P(e)$ and $F=$ $[P,\langle e\rangle]$. 
Since $e\neq 1$ it follows that the group $F$ is nontrivial. 
Since $P$ is abelian, by \cite[Theorem 4.34]{Isaacs2} or
by \cite[Theorem 5.2.3]{Gor}, we have 
$P=T\times F$. Using the formula \ref{conj-action-2} and that $P$ is abelian
we get that  conjugation action by  $g\in P$  on $kP\cdot e$ is equal to left 
multiplication with $g(^eg)^{-1}$. Thus, $T$ acts trivially on $kP\cdot e$ and $F$ acts
freely on $P\cdot e$. Using that $F$ is non-trivial, this implies $\H^1(F; kP\cdot e)=0$.
Furthermore, $kP\cdot e$ can be viewed as a $k(T\times F)$-module $k\tenk kP\cdot e$ 
where $k$ is the trivial $kT$-module, and $kP\cdot e$ is the free $kF$-module for
the conjugation action of $F$ on $kP\cdot e$. 
By the K\"unneth formula \ref{Kuenneth-groups-1}, we have
\[
\H^1(P;kP\cdot e)\cong \H^1(T\times F; k\tenk kP\cdot e)\cong 
\H^0(T;k)\tenk  \H^1(F; kP\cdot e)\oplus \H^1(T;k)\tenk \H^0(F;kP\cdot e)
\]
\[
\cong  \H^1(T;k)\tenk H^0(F;kP\cdot e) \cong \H^1(T;k)\tenk (kP\cdot e)^F.
\]
Any orbit in $P\cdot e$ under the conjugation action of $F$ is
equal to  $\{ta\cdot e: a\in F\}$, where $t\in T$.
Thus,  $(kP\cdot e)^F$ is spanned by the set
$\{ t(\sum_{a\in F} a)\cdot e: t\in T\}$.
It follows that the map $t \mapsto t(\sum_{a\in F} a)\cdot e$ induces an isomorphism
$$ kT \cong (kP\cdot e)^F,$$
where we note that $P$ acts trivially by conjugation on both sides. 
Since $E$ is abelian, the direct product decomposition $P=T\times F$ is stable
under the action of  $E$.  Thus the condition $[P,E]=P$ implies that  $[T,E]=T$ and 
$[F,E]=F$, so neither $T$ nor $F$ has order $2$. 
Since $F$ is non-trivial, so has order at least $3$, it follows that the socle element
$\sum_{a\in F} a$ of $kF$ is contained in $J(kF)^2$. Thus $(kP\cdot e)^F
\subseteq$ $kT\cdot J(kF)^2\cdot e\subseteq$ $J(kP)^2\cdot e$. It follows
that every class in $\HH^1(kP; kP\cdot e)^E$ has a representative with
image contained in $J(kP)^2\cdot e$. 
Let $D_2$ denote the image of $Der_2(k_\alpha(P\rtimes E))$ in 
$\HH^1(k_\alpha(P\rtimes E)$. 
By the above we have 
$\bigoplus_{e\in E\setminus\{1\}} \H^1(P;kP\cdot e)^E\subseteq D_2$.  Thus 
\begin{Statement} \label{HH1D2-statement}
\[  \HH^1(k_\alpha(P\rtimes E))= \HH^1(kP)^E+ D_2 .  \]
\end{Statement}
Theorem \ref{E-stable-derivations-thm} implies therefore that  every class in 
$\HH^1(k_\alpha(P\rtimes E))$ 
is represented by a derivation  with image contained in $J(k_\alpha(P\rtimes E))$.
Equivalently, the map
$\Der_1(k_\alpha(P\rtimes E))\to \HH^1(k_\alpha(P\rtimes E))$ is surjective.
By Corollary \ref{rad-Cor} (iv),
\begin{Statement}
we have that $\HH^1(k_\alpha(P\rtimes E))$ is solvable if 
and only if  $\HH^1(kP)^E$ is solvable. 
\end{Statement}
 Note that the action of $E$ on $\HH^1(kP)$ by taking conjugation in the twisted 
 group algebra does not depend on $\alpha$ by Statement \ref{conj-action-2}. 
 Thus, this action on $\HH^1(kP)$ coincides with the action of $E$ induced by the 
 usual conjugation of $E$ on $P$. The rest follows from Theorem
 \ref{E-stable-derivations-thm}.
\end{proof}

\begin{Remark}
The proof above can be refined to give a slightly stronger result. Using an 
isomorphism from \cite[Section 3.4]{Evens}, 
for $e$ a nontrivial element in $E$, we have 
$\H^1(T;k)\tenk (kP\cdot e)^F\cong$ $\H^1(T; (kP\cdot e)^F)$.
By the above proof, this is isomorphic to 
$\H^1(T; kT)\cong$ $\HH^1(kT)$. 
By Lemma \ref{Pe-cor-1}, every class in $\HH^1(kT)^E$ has a representative $d$ with 
image in $J(kT)$, i.e every class in $\H^1(T;kT)^E$ is represented by a $1$-cocycle 
$f\in Z^1(T; kT)$ with image contained in $J(kT)$. Mapping $f$ to the corresponding
$1$-cocycle $f'$ in $Z^1(T; (kP\cdot e)^F)$, the proof
 above shows that the  image of  $f'$ is contained  in 
 $J(kT) J(kF)^2\cdot e\subseteq$ $J(kP)^3\cdot e$. 
\end{Remark}

As mentioned in the Introduction, Theorem \ref{twisted-PE-thm} admits the following
equivalent reformulation.

\begin{Theorem} \label{central-PE-thm}
Let $P$ be a non-trivial finite abelian $p$-group and $E$ be a $p'$-group acting on $P$. 
Suppose that $[P,E]=P$, $C_E(P)\leq \mathrm{Z}(E)$ and that $E/C_E(P)$ is abelian. 
Every class in $\HH^1(k(P\rtimes E))$ is represented by a derivation on 
$k(P\rtimes E)$ with image contained in the Jacobson radical 
$J(k(P\rtimes E))$.
If the semisimple  $\Fp E$-module $P/\Phi(P)$ is multiplicity free, then
the Lie algebra $\HH^1(k(P\rtimes E))$ is solvable. The converse holds if $p$ is odd.
\end{Theorem} 

One way to prove Theorem \ref{central-PE-thm} would be to play it back to Theorem 
\ref{twisted-PE-thm}
via the standard correspondence between second cohomology  and central
group extensions. For convenience, we will present for convenience a self-contained 
proof which  directly translates
the proof of Theorem \ref{twisted-PE-thm} to the situation of Theorem 
\ref{central-PE-thm}.

\begin{proof}[{Proof of Theorem \ref{central-PE-thm}}]
Let $P$ be a non-trivial finite abelian $p$-group and $E$ be a $p'$-subgroup acting 
on $P$.  Set $C=$ $C_E(P)$. Suppose that $[P,E]=P$, $C\leq \mathrm{Z}(E)$ 
and  that $E/C$ is abelian. In particular, since $C$ acts trivially on $P$, $C$ also lies in 
the center of $P\rtimes E$. 
As before, we will use the canonical isomorphisms
$\HH^1(k(P\rtimes E))\cong$ $\H^1(P\rtimes E;k(P\rtimes E))\cong$
$ \H^1(P;k(P\rtimes E))^E.$
As a $kP$-module with respect to the conjugation action, we have 
$k(P\rtimes E)\cong \oplus_{e\in E}\ kP\cdot e$. 
Thus, we get $\HH^1(k(P\rtimes E))\cong$ $\H^1(P;\oplus_{e\in E}\ kP\cdot e)^E$
Since $C$ is a central $p'$-subgroup of $P\rtimes E$, and since for every $c\in C$ 
we have  $kP\cdot c\cong kP$ as a  $kP$-module, it follows that 
$\H^1(P; kP\times C) \cong \H^1(P; kP)\tenk kC$.
This space is $E$-invariant, with $E$ acting trivially on $kC$, and hence
$$\H^1(P; k(P\times C))^{E}\cong \H^1(P; kP)^{E}\tenk kC\cong
\HH^1(kP)^E\tenk kC. $$
Using the isomorphism \ref{HH-normal-2} (with trivial $\alpha$) it follows that 
$$\HH^1(k(P\rtimes E)) \cong \HH^1(kP)^E \tenk kC \ \oplus 
\ (\oplus_{e\in E\setminus C}\ \HH^1(kP; kP\cdot e))^E.$$
In particular, $\HH^1(kP)^E\tenk kC$ is a Lie subalgebra of $\HH^1(k(P\rtimes E))$. 
In order to analyse the second summand 
let $e\in E\setminus C$.
Set $T=C_P(e)$ and $F=$ $[P,\langle e\rangle]$. 
Since $e\not\in C$ it follows that the group $F$ is nontrivial. 
Since $P$ is abelian, by \cite[Theorem 4.34]{Isaacs2} or
by \cite[Theorem 5.2.3]{Gor}, we have 
$P=T\times F$. Notice that for every $g\in P$, the conjugation action of $g$ 
on $kP\cdot e$ is equal to  left 
multiplication with $g(^eg)^{-1}$. Thus, $T$ acts trivially on $kP\cdot e$ and 
$F$ acts freely on $P\cdot e$. Using that $F$ is nontrivial, this implies 
$\H^1(F; kP\cdot e)=0$.
Furthermore, $kP\cdot e$ can be viewed as a $k(T\times F)$-module 
$k\tenk kP\cdot e$  where $k$ is the trivial $kT$-module, and $kP\cdot e$ 
is the free $kF$-module for the conjugation action of $F$ on $kP\cdot e$. 
Just as in the proof of Theorem \ref{twisted-PE-thm}, by the K\"unneth formula 
\ref{Kuenneth-groups-1}, we have
$$\H^1(P;kP\cdot e)\cong \H^1(T;k)\tenk (kP\cdot e)^F,$$
and  the map $t \mapsto t(\sum_{a\in F} a)\cdot e$ induces an isomorphism
$$ kT \cong (kP\cdot e)^F.$$
Since $E/C$ is abelian, or equivalently, since the image of $E$ in $\Aut(P)$ is abelian,
it follows that the direct product decomposition $P=T\times F$ is still stable 
under the action of  $E$.  An obvious adaptation of the last part of the proof
of Theorem \ref{twisted-PE-thm} yields that 
$[F,E]=F\neq 1$ has order at least $3$, so as before the socle element
$\sum_{a\in F} a$ of $kF$ is contained in $J(kF)^2$, whence  $(kP\cdot e)^F
\subseteq$ $kT\cdot J(kF)^2\cdot e\subseteq$ $J(kP)^2\cdot e$. It follows
that every class in $\H^1(P;kP\cdot e)$ has a representative with
image contained in $J(kP)^2\cdot e$, so the  same holds for 
$\H^1(P;\oplus_{e\in E\setminus C}\ kP\cdot e)^E.$ Thus, every class in
 $\HH^1(kP; \oplus_{e\in E\setminus C}\ kP\cdot e)^E$ has a representative 
 with image contained in $J(kP)^2\cdot e$. Let $D_2$ denote the image of 
 $\Der_2(k(P\rtimes E))$ in  $\HH^1(k(P\rtimes E)$. 
By the above, we have 
$\HH^1(kP;\oplus_{e\in E\setminus C}\ kP\cdot e)^E\subseteq D_2$,
hence 
\begin{Statement}
$$\HH^1(k(P\rtimes E))=  \HH^1(kP)^E \tenk kC+ D_2.$$
\end{Statement}
Theorem \ref{E-stable-derivations-thm} implies therefore that  every class in 
$\HH^1(k(P\rtimes E))$ 
is represented by a derivation with image contained in $J(k(P\rtimes E))$.
Equivalently, the map $\Der_1(k(P\rtimes E))\to \HH^1(k(P\rtimes E))$ is surjective. 
By Corollary \ref{rad-Cor} (iv), 
\begin{Statement}
we have that the Lie algebra $\HH^1(k(P\rtimes E))$ is solvable if 
and only if  $\HH^1(kP)^E \tenk kC$ is solvable.  
\end{Statement}
The formula \ref{KuennethLie1} 
implies that this is the case if and only if $\HH^1(kP)^E$ is solvable.
The rest follows from Theorem \ref{E-stable-derivations-thm}.
\end{proof}

\begin{Corollary} \label{normal-defect-cor}
Let $G$ be a finite group, $B$ a block with a non-trivial abelian defect group $P$ 
and an abelian inertial quotient $E$ such that there is a stable equivalence of
Morita type between $B$ and its Brauer correspondent. Assume that $k$ is
large enough. Suppose that $[P,E]=P$.
If the $\Fp E$-module $P/\Phi(P)$ is multiplicity free, then $\HH^1(B)$ is a
solvable Lie algebra. The converse holds if $p$ is odd.
\end{Corollary}

\begin{proof}
Since the Lie algebra $\HH^1(B)$ of a block $B$ of a finite group algebra $kG$ is
preserved under stable equivalences of Morita type (cf. \cite[Theorem 10.7]{KLZ}), 
we may assume that $P$ is normal in $G$.  Arguing as before,
by  \cite[Theorem A]{Kuenormal}  the block $B$ is Morita equivalent to 
$k_\alpha(P\rtimes E)$ for some $\alpha\in$
$Z^2(E,k^\times)$, inflated to $P\rtimes E$ via the canonical surjection 
$P\rtimes E\to E$. Thus $\HH^1(B)\cong$  $\HH^1(k_\alpha(P\rtimes E))$ as Lie 
algebras. The result follows from Theorem \ref{twisted-PE-thm}.
\end{proof}

We remark that Brou\'e's abelian defect group conjecture predicts that there
should always be a derived equivalence between a block with an abelian
defect group and its Brauer correspondent, so in particular there should always
be a stable  equivalence of Morita type in that situation.

\section{Examples} 

\begin{Example} \label{cyclic-example}
Let $k$ be a field of prime characteristic $p$.
Let $P$ be a non-trivial cyclic $p$-group and $E$ a $p'$-subgroup of $\Aut(P)$.
Then $E$ is cyclic of order dividing $p-1$ and acts freely on $P\setminus \{1\}$.
It follows for instance from \cite[Example 5.7]{LiRuHH1} as well as 
explicit calculations in \cite[Section 3]{Murphy23}
that $\HH^1(k(P\rtimes E))$ is solvable if and only if $E\neq 1$ or if $p=2$, and 
$\HH^1(k(P\rtimes E))$ is 
simple if and only if $|P|=p\geq 3$ and $E=1$.  It is easy to see that this follows also 
from combining the results of this paper. If $E\neq 1$, then $\HH^1(k(P\rtimes E))$ is
solvable by Theorem \ref{twisted-PE-thm}, or also by Theorem
\ref{frobenius-block-thm}. If $|P|=p\geq 3$, then $\HH^1(kP)$ is a simple Witt Lie
algebra, if $|P|=p=2$, then $\HH^1(kP)$ is a solvable  non-abelian $2$-dimensional 
Lie algebra, and if $P$ has order $p^a$ for some $a\geq 2$, then $\HH^1(kP)$ is not 
simple, by Lemma \ref{lem1}, in conjunction with the fact that $\HH^1(kP)$ has dimension
$|P|$. Note that this describes the Lie algebra structure of $\HH^1(B)$ for $B$ a
block of a finite group algebra $kG$ with a non-trivial cyclic defect group $P$ and
intertial quotient $E$, with $k$ sufficiently large. Indeed, there is a stable
equivalence of Morita type between $B$ and $k(P\rtimes E)$ (see e.g. 
\cite[Theorem 11.1.2]{LiBookII}), and hence $\HH^1(B)\cong$ $\HH^1(k(P\rtimes E))$
by \cite[Theorem 10.2]{KLZ}. One can show this, of course, also by making use
of the fact, due to Rickard \cite{Rick89a}, that $B$ and $k(P\rtimes E)$ are in 
fact derived equivalent.
\end{Example}

\begin{Example} \label{C2-example} 
Let $p$ be an odd prime and $P$ a finite  abelian $p$-group of rank
$r\geq 1$. Let $E$ be a group of order $2$, acting on $P$, with the nontrivial element of
$E$ acting by inversion. 
This action of $E$ on $P\setminus \{1\}$ is free, so $P\rtimes E$ is a Frobenius group, and 
hence we have a Lie algebra isomorphism 
$$\HH^1(\Fp(P\rtimes E))\cong \Der(\Fp P)^E.$$
Set $J=J(\Fp P)$.
The nontrivial element of $E$ acts as inversion on $P$, hence also on $P/\Phi(P)$, and 
therefore  as multiplication by $-1$ on  $J/J^2$ via the isomorphism from Lemma 
\ref{PhiP-Lemma}.  Since this action is in the center of $\End_{\Fp}(J/J^2)$, it follows 
that  $\End_{\Fp E}(J/J^2)=$ $\End_{\Fp}(J/J^2)$.
 Lemma \ref{lem4} implies that   the canonical Lie algebra homomorphism
 $$\HH^1(\Fp(P\rtimes E))\cong \Der(\Fp P)^E \to \End_{\Fp E}(J/J^2)\cong \gl_r(\Fp)$$
 is surjective.
 In particular, if $r\geq 2$, then the Lie algebra $\HH^1(\Fp(P\rtimes E))$ is not solvable,
 while for $r=1$ this is a solvable Lie algebra of dimension $\frac{|P|-1}{2}$. 
 \end{Example}

\begin{Example} \label{split-example}
Let $P$ be a finite elementary abelian $p$-group of rank $r\geq 1$, and let $E$ be a 
$p'$-subgroup of $\Aut(P)$. Suppose that $P$ is split semisimple as an 
$\Fp E$-module. Then $P$ has an $E$-stable decomposition 
$P=\prod_{i=1}^r \langle x_i \rangle$, where the $x_i$ all have order $p$;
in particular, $E$ is abelian.
If the subgroups $C_E(x_i)$, with $1\leq i\leq r$, are pairwise distinct proper subgroups 
of $E$, then the simple $\Fp E$-modules $\langle x_i\rangle$ are pairwise non-isomorphic 
non-trivial. In that case, setting $E_i=E/C_E(x_i)$, each $E_i$ is non-trivial, and the canonical surjections $E\to E_i$ induce an  isomorphism $E\cong \prod_{i=1}^r E_i$. Note that each 
$E_i$  can be identified with  a non-trivial cyclic  subgroup of $\Aut(\langle x_i\rangle)$, 
and hence $P\rtimes E = \prod_{i=1}^r \langle x_i\rangle \rtimes E_i$.
In particular, $P$ is multiplicity free as an $\Fp E$-module, and $P=$ $[P,E]$. Thus 
Theorem \ref{E-stable-derivations-thm} implies that $\HH^1(kP)^E$ is a solvable 
Lie algebra. In fact, 
since $P\rtimes E$ is a direct  product of Frobenius groups, one can, using the 
K\"unneth formula, even conclude that   $\HH^1(k(P\rtimes E))$ is solvable; see 
Example \ref{Kuenneth-example} below.
\end{Example}

\begin{Example} \label{CpCp-example}
Let $k$ be an algebraically closed  field of characteristic $p$,  let $G$ be a finite 
group and $B$ a block of $kG$ with an elementary abelian defect group
$P\cong $ $C_p\times C_p$ of order $p^2$ and an abelian  inertial quotient $E$. 
We identify $E$ with a subgroup of $\GL_2(\Fp)$ when convenient.

The Brauer correspondent block with defect group $P$ is source algebra
equivalent to $k_\alpha(P\rtimes E)$ for some $\alpha\in$ $Z^2(E;k^\times)$.
By a result of Rouquier \cite[6.3]{Rouq01} (see \cite[Theorem A.2]{Liperm} for a
proof for non-principal blocks), there is a stable equivalence of Morita type
between $B$ and $k_\alpha(P\rtimes E)$, and hence, by 
\cite[Theorem 10.2]{KLZ}, we have an isomorphism of Lie algebras 
$$\HH^1(B) \cong \HH^1(k_\alpha(P\rtimes E)).$$

If $E=1$, or equivalently, if the block $B$ is nilpotent, 
then $\HH^1(B)\cong$ $\HH^1(kP)$ is a simple Witt Lie algebra
(see \cite[Theorem 1.1]{LiRusimple} and the references in that paper for background).

If $E\neq 1$ and $E\leq Z(\GL_2(p))$, then $p$ is odd, we have $[P,E]=P$,
and  $P$ is a direct sum of  two isomorphic simple $\Fp E$-modules. Hence, in that 
case, $\HH^1(B)$ has a  quotient isomorphic to  $\gl_2(\Fp)$, 
so is non-solvable and non-simple,  where we use Lemma \ref{lem4}.

If $E\neq 1$ and $[P,E]<P$, then $P\rtimes E\cong$ $Q\times (R\rtimes E)$,
where $Q=C_P(E)$ and $R=[P,E]$ both have order $p$ and $E$ is cyclic of order
dividing $p-1$ (so $p$ is odd and $\alpha$ is trivial). 
Since $\HH^1(kQ)$ is a simple Witt Lie algebra and every class in 
$\HH^1(k(R\rtimes E))$ is represented by a derivation of $k(R\times E)$ with 
image contained in $J(k(R\times E))$, it follows from Lemma
\ref{KuennethRad2} that $\HH^1(B)$ is neither simple nor solvable in that case.

If $E$ is a non-trivial abelian but not a central subgroup of $\GL_2(\Fp)$, such
that $[P,E]=P$, then $P$ is either simple as an $\Fp E$-module  or the sum of two 
non-isomorphic simple $\Fp E$-modules, and hence $\HH^1(B)$ is solvable in that 
case, by Theorem \ref{twisted-PE-thm}. For instance, if  $E=$ $C_2\times C_2$
and $\alpha$ is non-trivial, then the Brauer correspondent of $B$ is Morita
equivalent to the quantum complete intersection desctibed in 
Example \ref{qi-example} below. Another instance arises as follows.
By identifying $P$ with the additive group of $\F_{p^2}$ one sees that
the cyclic group $\F_{p^2}^\times$ of order $p^2-1$ acts regularly on 
$P\setminus \{1\}$. Let $m\geq 3$ be a divisor of $p+1$ coprime to $p-1$,
and let $C_m$ a cyclic group of order $m$ acting freely on the non-identity 
elements of $P$. Then $[P,E]=P$ and  $P$ is a simple $\Fp E$-module, hence  
$\HH^1(k(P\rtimes E))\cong$ $\HH^1(kP)^E$
is a solvable Lie algebra by Theorem \ref{frobenius-2-thm}. 
\end{Example}

\begin{Example} \label{Kuenneth-example}
Theorem \ref{frobenius-2-thm} can be extended to direct products of Frobenius groups, 
using the K\"unneth formula. For $1\leq  i \leq n$ let $P_i$ be a non-trivial finite abelian 
$p$-group and let $E_i$ be a $p'$-subgroup of $\Aut(P_i)$ acting freely on 
$P_i\setminus \{1\}$. By Proposition \ref{twisted-PE-prop} we have
$\HH^1(k(P_i\rtimes E_i))\cong$ $\HH^1(kP_i)^{E_i}$, so by Lemma \ref{Pe-cor-1} 
the canonical  map $\Der_1(k(P_i\rtimes E_i))\to$ $\HH^1(k(P_i\rtimes E_i))$ 
are surjective.

Set $P=\prod_{i=1}^n P_i$ and $E=\prod_{i=1}^n E_i$.
Proposition \ref{Kuenneth-prop}, applied repeatedly, implies that the map 
$\Der_1(k(P\rtimes E))\to$ $\HH^1(k(P\rtimes E))$ is surjective. 
Suppose in addition that $P/\Phi(P)$ is multiplicity free as an $\Fp E$-module, or 
equivalently, that each $P_i/\Phi(P_i)$ is multiplicity free as an $\Fp E_i$-module, 
for $1\leq i\leq n$.  Lemma  \ref{lem4} implies that 
$$[\Der_1(k(P\rtimes E)), \Der_1(k(P\rtimes E))]\subseteq 
\Der_2(k(P\rtimes E)) \ + \IDer(k(P\rtimes E)),$$ 
and in particular, that $\HH^1(k(P\rtimes E))$ is a solvable Lie algebra.
In the case where $E$ is abelian, this follows also from Theorem \ref{twisted-PE-thm}. 
\end{Example}

\begin{Example} \label{qi-example}
Suppose that $p$ is an odd prime.  Consider the algebra 
$$A= k\langle x,y|  x^p=y^p=0,\ xy+yx=0\rangle .$$
This algebra, which has dimension $p^2$,  arises as basic algebra of a the non-principal 
block  of  $k(C_p\times C_p)\rtimes Q_8$, with defect group $C_p\times C_p$ and 
inertial  quotient $C_2\times C_2$. It is shown in \cite[Theorem 1.1]{BKL} amongst 
other  statements on the structure of the Lie algebra $\HH^1(A)$ that $\HH^1(A)$ is
a solvable Lie algebra. Theorem \ref{twisted-PE-thm} provides an alternative proof 
of this fact.
\end{Example}

\begin{Example} \label{KoKu-example}
Let $k$ be an algebraically closed field of prime characteristic $p$, let $G$ be
a finite group, $N$ a normal subgroup such that $G/N$ is a $p$-group, and
let $b$ be a $G$-stable block of $kN$. Set $C=kNb$. Then $B=kGb$ is a block 
of $kG$. Let $P$ be a defect group of $B$, and set $Q=P\cap N$. Suppose that $P$
is abelian and that $Q$ has a complement $R$ in $P$. If $p$ is odd or if $R$ has
rank at least $2$, then $\HH^1(B)$ is not solvable. Indeed, by a Theorem of
Koshitani and K\"ulshammer in \cite{KoKu} (see  \cite[Theorem 10.4.2]{LiBookII}
for an expository account), we have a $k$-algebra isomorphism $B\cong$ 
$kR\tenk C$, so by the K\"unneth formula \ref{Kuenneth}, $\HH^1(B)$ has
a Lie subalgebra isomorphic to $\HH^1(kR)$, which in turn has a Lie algebra
quotient $\HH^1(kR/\Phi(R))$. If $p$ is odd or if $R$ has rank at least $2$, then
$\HH^1(kR/\Phi(R))$ is a simple Witt Lie algebra.
\end{Example}

\bigskip\noindent
{\bf Acknowledgements.} 
The authors  acknowledge support from
EPSRC grant EP/X035328/1.

\end{document}